\documentclass[12pt]{article}

\usepackage[utf8]{inputenc}
\usepackage[T1]{fontenc}
\usepackage{lmodern}       
\usepackage{graphicx}%
\usepackage{multirow}%
\usepackage{amsmath,amssymb,amsfonts}%
\usepackage{amsthm}%
\usepackage{mathrsfs}%
\usepackage[title]{appendix}%
\usepackage{xcolor}%
\usepackage{textcomp}%
\usepackage{manyfoot}%
\usepackage{booktabs}%
\usepackage{enumitem}
\usepackage{algorithm}%
\usepackage{algorithmic}%
\usepackage{listings}%
\usepackage{geometry}       
\usepackage{url}
\usepackage[numbers,sort&compress]{natbib}

\usepackage{mathtools}
\usepackage{bbm}
\usepackage{bm}

\usepackage[hidelinks]{hyperref}
\usepackage[capitalize,noabbrev]{cleveref}
\crefname{equation}{eq.}{eqs.}
\usepackage{autonum}

\newtheorem{theorem}{Theorem}
\newtheorem{proposition}[theorem]{Proposition}%
\newtheorem{lemma}[theorem]{Lemma}

\theoremstyle{remark}%
\newtheorem{remark}{Remark}%

\theoremstyle{definition}%
\newtheorem{definition}{Definition}%

\newcommand{\simulparen}[1]{\left\{ \begin{aligned} #1 \end{aligned} \right.}

\DeclarePairedDelimiter\paren{\lparen}{\rparen}
\DeclarePairedDelimiterX{\inpr}[2]{\langle}{\rangle}{{#1},{#2}}
\DeclarePairedDelimiterX{\setI}[2]{\{}{\}}{\,{#1}\ \delimsize| \ {#2}\,}
\DeclarePairedDelimiter{\setE}{\{}{\}}
\DeclarePairedDelimiter{\abs}{|}{|}
\DeclarePairedDelimiter{\norm}{\|}{\|}

\newcommand{\RR}{\mathbb{R}}
\newcommand{\CC}{\mathbb{C}}
\newcommand{\dd}{\mathrm{d}}
\newcommand{\x}[1]{x^{\paren*{#1}}}

\newcommand{\Order}[1]{\mathrm{O} \paren*{#1}}

\newcommand{\dv}[2]{\frac{\dd #1}{\dd #2}}
\newcommand{\pdv}[2]{\frac{\partial #1}{\partial #2}}
\newcommand{\e}{\mathrm{e}}

\title{Essential Convergence Rates of Continuous-Time Models for Optimization Methods\thanks{%
A preliminary version of this paper appeared in JSIAM Lett., 14 (2022), pp. 119–122. https://doi.org/10.14495/jsiaml.14.119
}}

\author{Kansei Ushiyama\thanks{Department of Mathemathical Informatics,
Graduate School of Information Science and Technology,
The University of Tokyo, Tokyo 113-8656, Japan
  (\texttt{ushiyama-kansei074@g.ecc.u-tokyo.ac.jp}, \url{https://kanseiushiyama.github.io/}).}
\and Shun Sato\thanks{Department of Mathematical Sciences, Graduate School of Science,
Tokyo Metropolitan University, Tokyo 192-0397, Japan (\texttt{shun\_sato@tmu.ac.jp}).}
\and Takayasu Matsuo\thanks{Department of Mathemathical Informatics,
Graduate School of Information Science and Technology,
The University of Tokyo, Tokyo 113-8656, Japan
  (\texttt{matsuo@mist.i.u-tokyo.ac.jp}).}}

\date{}

\allowdisplaybreaks[4]
\begin{document}

\maketitle

\begin{abstract}
Designing and analyzing optimization methods via continuous-time models expressed as ordinary differential equations (ODEs) is a promising approach for its intuitiveness and simplicity. A key concern, however, is that the convergence rates of such models can be arbitrarily modified by time rescaling, rendering the task of seeking ODEs with ``fast'' convergence meaningless. To eliminate this ambiguity of the rates, we introduce the notion of the {\em essential convergence rate}. We justify this notion by proving that, under appropriate assumptions on discretization, no method obtained by discretizing an ODE can achieve a faster rate than its essential convergence rate.
\end{abstract}



\section{Introduction}
Some optimization methods can be modeled as ordinary differential equations (ODEs) by taking the limit as their step sizes vanish, as exemplified by the well-known relationship between the steepest descent method and gradient flow. This perspective has facilitated both the derivation and analysis of optimization methods, while also providing physical intuition into the optimization process, as demonstrated in Polyak’s seminal work~\cite{P63} on the heavy-ball method.

This approach gained significant attention following the introduction of the continuous time limit of Nesterov's accelerated gradient method (AGM)~\cite{N83} by Su et al.~\cite{SBC14}, which offered a novel viewpoint on the underlying mechanisms of acceleration. Since then, numerous ODEs have been proposed to model various accelerated methods, such as the triple momentum method (TMM)~\cite{VFL18} and the information-theoretic exact method (ITEM)~\cite{TD22}, as discussed in~\cite{KY23b}, as well as ODEs derived via alternative continuous-limit techniques, such as high-resolution ODEs~\cite{SDJS22}.
This line of work has now become a powerful framework for understanding and designing optimization methods, as surveyed in~\cite{ADA21}, with notable advances including~\cite{BCL20b,BCL21,HHF21,L22,ACFR222,DJ21,SPR23,SP24b}.

The efficiency of ODE models is characterized by their convergence rates, which have been analyzed for various ODEs under different settings. Consistency between the convergence rates of ODE models and those of the corresponding discrete optimization methods provides valuable insight into algorithm behavior. Moreover, identifying ODE models with fast convergence rates can guide the development of efficient optimization algorithms, as demonstrated in~\cite{KBB15}.

Numerous frameworks for analyzing the convergence rates of ODE models have been proposed, e.g., in~\cite{DO19,SRR22,MTB23,KY23b,USM24,BPRS24}. Since convergence proofs for ODEs are often more straightforward than for their discrete counterparts, this approach offers clearer understanding of convergence proofs through analogies between continuous-time and discrete-time analyses~\cite{USM23}.

However, an efficient continuous-time model cannot be identified solely by comparing convergence rates, since these rates can be arbitrarily modified through time-rescaling without changing the solution trajectory, i.e., the underlying convergence mechanism. Consequently, this arbitrariness of the rates renders the question ``what is the fastest ODE?'' meaningless.

For example, if $f\colon \RR^d \to \RR$ is convex, the gradient flow \[
\dot{x}(t) = - \nabla f(x(t)),\quad x(0)=x_0 \in \RR^d,\] where $\dot{x}$ denotes the time derivative of $x$, achieves the rate $f(x(t)) - f^\star = \Order{1/t}$ (cf.~\cite{SBC16}), where $f^\star$ is the optimal value. However, a dynamics attaining a faster rate $\Order{1/t^2}$ can be obtained only by considering the time-rescaled gradient flow $\dot{x}(t) = -t\nabla f(x(t))$. 
The same rate is also achieved by the solution $x(t)$ of the system
\begin{equation}
       \ddot{x}(t) + \frac{3}{t}\dot{x}(t) + \nabla f(x(t)) = 0,\quad x(0)=x_0 \in \RR^d\label{example}
\end{equation}
which corresponds to the continuous-time limit of AGM~\cite{SBC14}.
Thus, one cannot determine which of these dynamics is superior solely on the basis of their convergence rates.

This observation highlights the necessity of a tool for selecting a representative rate for each ODE. 
In~\cite{WWJ16}, the authors also emphasized the arbitrariness of convergence rates under time-rescaling, noting that “the underlying solution curve has a more fundamental structure that is worth exploring further.”

In prior studies on convergence rates of ODEs, the ambiguity caused by time-rescaling has often been addressed by fixing the coefficients of the ODEs, thereby restricting the collection of ODEs under comparison. While this approach enables comparison within a given collection of ODEs, it does not allow meaningful comparison across different collections. 
Moreover, even within the same collection, comparing ODEs with fixed coefficients may not yield valid insights.
For instance, consider the family of heavy-ball ODEs~\cite{P63}:
\begin{equation}
       \ddot{x}(t) + a_1(t) \dot{x}(t) + a_2(t) \nabla f(x(t)) = 0, \quad x(0)=x_0 \in \RR^d\label{HBODE}
\end{equation}
where $a_1, a_2 \colon [0, \infty) \to \RR$. An ODE in this family is specified by determining the coefficient functions $(a_1, a_2)$. We define an equivalence relation on these coefficients: $(a_1,a_2) \sim (b_1,b_2)$ if the solution of the ODE with $(a_1,a_2)$ is a time-rescaled one of the solution with $(b_1,b_2)$. This relation partitions the heavy-ball ODEs into equivalence classes.
Our goal is then to compare these classes and identify the one that produces the most efficient solution trajectories, which in turn correspond to efficient discrete-time algorithms.
Since for any $(a_1,a_2)$, there exists a unique $\hat{a}$ such that $(a_1,a_2) \sim (\hat{a},1)$, one can define a “normalized” representative in each class by fixing the coefficient $a_2(t)=1$.
Prior works have compared heavy-ball ODEs under this normalization~\cite{MJ20,ADR22,SRR22,MTB23}. However, there is no justification that comparisons restricted to these representatives reflect the relative efficiency of the underlying equivalence classes, i.e. the solution trajectories.

This ambiguity becomes problematic when the heavy-ball ODE involves a gradient evaluated at a shifted point, which is equivalent to Hessian-driven damping~\cite{SDJS22}.
Consider the following family of ODEs:
\begin{equation}
       \ddot{x}(t) + \frac{3}{t} \dot{x}(t) + \nabla f(x(t) + b\dot{x}(t)) = 0, \quad x(0)=x_0 \in \RR^d\label{ODEshifted}
\end{equation}
where $b > 0$ is a constant. For these ODEs, suppose that we eliminate the time-rescale ambiguity is eliminated by fixing the coefficient in front of the gradient term to 1. It has been shown~\cite{USM24} that, for convex $f$,
\begin{equation}\label{eq:rate:NAGS}
       \min_{0\le s \le t} \norm{\nabla f(x(s) + b\dot{x}(s))}^2 \le \frac{18\norm{x(0)-x^\star}^2}{7bt^3}.
\end{equation}
If we compare these solution trajectories solely based on this convergence rate, we might conclude that larger values of $b$ are better, suggesting an unlimited improvement in the constant factor of the rate.
However, such a phenomenon cannot occur in the discrete-time setting, given the known lower bounds (cf.~\cite{N18b}).
Thus, it is unlikely that a fast algorithm can be developed by discretizing this model.
We therefore need to adopt another approach to fixing the rate (or selecting the representative rate), so as to avoid continuous-time models that yield illusory convergence rates that do not correspond to any feasible algorithm.


A similar issue arises in control theory.
Increasing the gain enables faster convergence to the desired state, which corresponds to a time-rescaling effect in continuous-time models.
In practical control situations, however, excessively large gain leads to higher control cost and greater sensitivity to noise, making arbitrarily large gain impractical.
To balance this trade-off, optimal control strategies, such as the linear quadratic method (cf.~\cite{AM07b}), typically minimize a cost function that accounts for both performance and these side effects.
In this study, we focus on a similar trade-off structure arising from time-rescaling, although our strategy differs from the control-theoretic approach above.

We view optimization methods as numerical solvers for ODEs, where the allowable step size determines the effective convergence rate after discretization.
Our key observation is the following trade-off: in continuous-time models, time-rescaling can speed up convergence but at the same time forces smaller step sizes for stable numerical integration, thereby slowing discrete-time convergence.
Using linear stability analysis, a standard tool from numerical analysis, we show that this trade-off collapses to a single convergence rate, which serves as a representative rate of the ODE, as detailed in \cref{sec:essential}.

Building on this observation, we define the \emph{essential convergence rate} for each equivalence class comprising dynamics that can be transformed into one another via time-rescaling.
We then show that, under suitable assumptions on the numerical scheme, no rate faster than the essential convergence rate is attainable after discretization, indicating that this rate is the appropriate yardstick for assessing the true convergence behavior of continuous-time models.

\paragraph{Comparison with~\cite{USM2022rate}}
This paper extends our previous letter~\cite{USM2022rate}.
The preceding work~\cite{USM2022rate} also defined the essential convergence rate, but only for a fixed objective function and a fixed initial point.
In contrast, the present paper generalizes this notion to the worst-case setting, allowing both the objective function and the initial point to vary within a prescribed set.
Such a formulation provides a more meaningful measure of the efficiency of the dynamics.
Moreover, in~\cite{USM2022rate}, the rate was defined in a low-resolution manner: if the essential convergence rate was linear, then any linear rate of the form $\mathrm{e}^{-a t}$ with arbitrary $a>0$ qualified as the essential convergence rate.
In this paper, we refine the definition to a sufficiently high-resolution form, which enables us to distinguish, for instance, between the linear convergence rates of the ODE models corresponding to the gradient descent and the accelerated gradient method.


\paragraph{Organization}
The paper is organized as follows. Preliminaries from numerical analysis and the formulation of our problem are presented in \cref{sec:preliminary}. The definition of the essential convergence rate and its justification are given in \cref{sec:essential}. Examples illustrating essential convergence rates of certain dynamics are provided in \cref{sec:example}. Further discussion of the implications and limitations of this work is presented in \cref{sec:discussion}, and conclusions are drawn in \cref{sec:conclusions}.

\section{Preliminary}\label{sec:preliminary}
\subsection{Linear stability analysis}
To define the essential convergence rate, we draw on the concept of linear stability analysis of numerical methods for ODEs.
Once optimization methods are regarded as numerical methods of ODEs, the stability of the numerical scheme is a necessary condition for the convergence of optimization methods.

Here, We consider the ODE
\begin{equation}
       \dot{x}(t) = g(x(t),t), \quad x(0) = x_0 \in \RR^d, \label{ode}
\end{equation}
where $g \colon \RR^d \times \RR_{\ge 0} \to \RR^d$ is a vector field.
Let $\delta x(t) \in \RR^d$ denote a perturbation of the solution $x(t)$, and consider its evolution under the linearized equation
\begin{equation}
       \dv{}{t}\delta x(t) = \pdv{g}{x}(x(t),t) \delta x(t), \quad \delta x(0) = \delta x_0 \in \RR^d,
\end{equation}
where $\partial g/\partial x$ denotes the Jacobian of $g$.
When this linearized system is computed by a single step of a numerical method, the perturbation must not grow; otherwise the computation for~\eqref{ode} becomes unstable.
By examining the eigenvalue decomposition of $\partial g/\partial x$, we arrive at the following test equation.

\begin{definition}[Stability domain of a numerical method; cf.~\cite{HW10b}]
       Let $R(\lambda h)$ be the value obtained by applying the numerical method with a step size $h$ to Dahlquist's test equation
       $
              \dot{x} = \lambda x, \ x(0) = 1. \label{dahlquist}
       $ 
       The function $R$ is said to be the {\em stability function} of the method, and the set
       $
             S:= \setI{z \in \CC}{\abs{R(z)} \le 1}
       $
       is called the {\em stability domain} of the method.
\end{definition}

For a given ODE~\cref{ode} and a numerical method, the numerical computation proceeds stably (in the sense of linear stability) if the step sizes $\setE{h_k}$ are controlled so that, for any $k$, the products $h_k\lambda$ lie within the stability domain of the numerical method for all eigenvalues $\lambda$ of $(\partial g/\partial x)$.
Note that a number of numerical methods with various stability domains have been developed to date (cf.~\cite{HW10b}).
To allow for larger step sizes, it is important to employ a numerical method whose stability domain adequately reflects the configuration of the eigenvalues of the target problem. This consideration will also play a role in the optimization context, as will be seen in~\cref{thm}.

Because optimization methods are typically based on explicit schemes, whose stability domains are generally bounded (cf.~\cite{NS1974}), we henceforth restrict our attention to numerical methods with bounded stability domains.

\subsection{Problem setting and notations}
In the following, $\mathcal{F}$ denotes a general collection of objective functions $f \colon \RR^d \to \RR$ and $\mathcal{S}_{\mu,L}$ denotes the set of $\mu$-strongly convex and $L$-smooth functions on $\RR^d$ with respect to the Euclidean norm. For a matrix $A\in \RR^{d\times d}$, $\rho(A)$ denotes its spectral radius.

Let $f \in \mathcal{F}$ and we consider the unconstrained optimization problem
\begin{equation}
       \min_{x \in \RR^d} f(x).
\end{equation}
We assume that an (possibly local) optimal value $f^\star$ and an associated optimal solution $x^\star$ exist.
This paper focuses on ODE models whose solutions converge to the optimal solution, e.g.,~\eqref{example}.
To include second- or higher-order ODEs, we consider the following $d'$-dimensional ($d'=d, 2d, \ldots$) first-order non-autonomous continuous-time system generated by a mapping $g[\cdot] \colon \mathcal{F} \to (\RR^{d'} \times \RR_{\ge 0} \to \RR^{d'})$:
\begin{equation}
       \dot{y} = g[f](y,t),\quad y(0) = y_0\in \mathcal{V}, \label{ode4}
\end{equation}
where the set $\mathcal{V} \subseteq \RR^d$ denotes the set of initial values.
We set  $y_i = x_i$ for $i = 1,\dots,d$, and if $d'> d$ the remaining coordinates $y_{d+1},\dots,y_{d'}$ are interpreted as auxiliary variables introduced as needed.

We accordingly extend the objective function $f$ by defining $\tilde{f}(y) = f(y_1,\dots,y_{d'})$.
For simplicity, we will abuse notation and write $\tilde{f}$ as $f$ again.

Let $\mathcal{G}$ denote the collection of maps $g[\cdot]$'s such that, for any $f \in \mathcal{F}$ and any initial point $y_0 \in \mathcal{V}$, the ODE~\eqref{ode4} is globally well-posed, and its solution $y(t) \in \RR^{d'}$ satisfies $f(y(t)) \to f^\star$ as $t \to \infty$.
For a solution $y(t)$ of~\eqref{ode4}, ${\partial g[f]((y(t),t))}/{\partial y}$ denotes the Jacobian matrix of $g[f]$ with respect the first argument, evaluated at $(y(t),t)$.

In the following, a function $\alpha \colon \RR_{\ge 0} \to \RR_{\ge 0}$ is called a time-rescaling function if it is differentiable, monotonically increasing and satisfies $\alpha(0)=0$ and $\lim_{t\to\infty}\alpha(t)=\infty$.

\section{Essential convergence rate}\label{sec:essential}
We begin with an illustrative example.
Let us consider the one-dimensional gradient flow $\dot{x} = -\lambda x$ for $ \lambda \in [ \mu , L ] $, which corresponds to an eigenvalue component of the gradient flow $\dot{x} = - \nabla f(x) $ on $\RR^d$, where $f$ is a quadratic function belonging to $\mathcal{S}_{\mu, L}$.
For this class of functions, the convergence rate of the optimal gap $f(x(t)) - f^\star$ is $\Order{\e^{-2\mu t}}$ (see~\cref{app:rates_ie}), while the discrete rate by the gradient descent is $\Order{\e^{-4(\mu/(L+\mu))k}}$ (see~\cref{app:rates_ie}; advanced discussions are in~\cite{ADA21}).
In what follows, we consider time-rescaling of this ODE and the resulting convergence rate after discretization.

Suppose we rescale time by linear transformation $t = r\tau$ ($r > 0$), which yields the rescaled ODE
\[
\dot{x} = -r \lambda x.
\]
The convergence rate now becomes $\Order{\e^{-2r\mu \tau}}$, which appears accelerated if $r>1$.
But is this truly superior than the original ODE for designing discrete-time optimization methods?
To examine this, let us apply the explicit Euler method and see what would happen.
Its stability domain is the unit disc centered at $z=-1$ (see, for example,~\cite{HW10b}).
Since $-r \lambda$ (the eigenvalue of the Jacobian in this case) is real, it suffices to consider the real interval $[-2, 0]$ within the stability domain, which requires decreasing step sizes $h_{k} \le 2/(r\lambda)$ for any $\lambda \in [\mu,L]$.
This implies that the accumulated discrete time $\tau_k = \sum_{i=1}^k h_k$ cannot exceed $2k /(rL)$, and that with $\tau = 2k /(rL)$, the attainable discrete rate derived from the rescaled ODE cannot be faster than $\Order{\e^{-4(\mu/L)k}}$.
(This rate is almost exact; see~\cref{app:rates_ie}.)
Thus, the time-rescaling factor $r$ cancels out and has no effect on the expected discrete rate.

Next, let us examine a more aggressive rescalingby $t=\tau^p$ ($p>1$), which yields
\[
\dot{x} = -p\tau^{p-1} \lambda x.
\]
This seems to improve the original continuous-time convergence rate from $\Order{\e^{-2\mu t}}$ to $\Order{\e^{-2\mu \tau^p}}$.
However, for the same reason as above, we are forced to take step sizes $h_k \le 2/(p{\tau_{k-1}}^{p-1} \lambda)$, which implies that the accumulated discrete time becomes $\tau_k = \sum_{i=1}^k h_k \sim (2k/L)^{1/p}$.
(This can be checked by the rough estimate $\sum_{j=1}^k 1/j^{(p-1)/p} \sim \int_1^k 1/{x^{(p-1)/p}}{\rm d}x \sim k^{1/p}$.) As a result, the apparent acceleration is canceled out, and we come back to the rate $\Order{\e^{-4 (\mu/L) k}}$.
This example also shows that time-rescaling does not essentially improve the order of the convergence.

Summing up, these toy problems show that {\em time-rescaling is meaningless once stable discretizations are taken into account}.
This immediately raises another question: can we somehow appropriately define a representative rate for an ODE that is useful for comparing different ODEs?

The linear case above provides an insight.
Recall that the coefficient of the discrete rate $4\mu/L$ came from $2r\mu \cdot 2/(rL)$, which consists of the three elements: (i) the continuous rate coefficient $2r\mu$, (ii) the ODE coefficient $rL$, and (iii) the radius of the stability domain of the explicit Euler method, namely $2$.
The third element solely depends on the choice of numerical method and has therefore no relation to the ODE.

Thus, the ratio (i)/(ii)${}=2r\mu/(rL)=2(\mu/L)$ and its associated rate $\Order{\e^{-2(\mu/L)t}}$ can be regarded as intrinsic to the ODE.
This rate corresponds to that of the ODE $\dot{x} = -\frac{\lambda}{L}x$, where the time scale is chosen so that the absolute value of the coefficient of the vector field is at most $1$ (i.e., we take $r=1/L$).
This observation suggests that {\em it can be meaningful to choose a time scale in which the asymptotic spectral radius of the Jacobian of the vector field (equivalently, the Hessian of the objective function) is at most $1$, and to regard the corresponding rate as the representative rate.}
 
We now rephrase the above procedure in more general setting to establish our claim with a mathematical rigor.
Consider a time-rescaling of ODE~\eqref{ode4}.
With the time-rescaling $t=\alpha(\tau)$, ODE~\eqref{ode4} becomes
\begin{equation}
       \dv{}{t}y(\alpha(t)) = \dot{\alpha}(t)g[f](y(\alpha(t)),\alpha(t)), \label{ode4:trans}
\end{equation}
where, for simplicity, we again denote $\tau$ by $t$.
We next apply numerical methods to these ODEs.
As stated previously, only numerical methods with bounded stability domains will be considered. This implies that, in~\eqref{ode4:trans}, all the eigenvalues of $h_k\dot{\alpha}(t)(\partial g[f]/\partial y)$ must remain within the stability domain. Therefore, if for example $\dot{\alpha}(t)\rho(\partial g[f]/\partial y)\to\infty$ as $t\to\infty$, we are forced to take decreasing step sizes $h_k\to 0$. In such cases, overall efficiency may not improve, as illustrated in the example above.
This observation motivates the following equivalence relation.
\begin{definition}
       For $g_1,g_2 \in \mathcal{G}$, the symbol $g_1 \sim g_2$ means that there exists a time-rescaling function $\alpha$ satisfying $g_1[f](y,t) = \dot{\alpha}(t) g_2[f](y,\alpha(t))$ 
       for any $t \in \RR_{> 0}$, $y \in \RR^d$, $f \in \mathcal{F}$.
       The symbol defines an equivalence relation on $\mathcal{G}$, and we denote the equivalence class for $g \in \mathcal{G}$ by $[g]$.
\end{definition}

As seen in the previous example, 
within an equivalence class $[g]$, 
it is essential to consider the representative for which $\dot{\alpha}(t)\rho(\partial g[f]/\partial y)=\Theta(1)$, so that simple fixed time-stepping schemes can be employed and the order of the continuous-time rate (e.g., $\Order{1/t}$ for gradient flow on convex functions) naturally coincides with that of the discrete method (e.g., $\Order{1/k}$).
This motivates the following concept.

In what follows, we assume $g$ and its equivalent class $[g]$ are given.
For a function $f \in \mathcal{F}$ and an initial value $y_0 \in \mathcal{V}$, 
we denote by $y_{g[f], y_0}(t)$ the solution of $\dot{y} = g[f](y,t)$ with $y(0) = y_0$.

\begin{definition}[$c$-essential dynamics]\label{def:ed}
    Suppose $g\in[g]$, $f \in \mathcal{F}$, and $y_0 \in \mathcal{V}$ are given.
       Let
       \begin{equation}
              \mathcal{L}_g \coloneqq \setI{(f,y_0) \in \mathcal{F} \times \mathcal{V}}{\lim_{t\to\infty} \rho\left({\partial g[f](y_{g[f],y_0}(t),t)}/{\partial y} \right) \text{ exists}}.
       \end{equation}
    Then, $g$ is said to be {\em $c$-essential} if it satisfies $\mathcal{L}_g \neq \emptyset$ and
       \begin{align} \label{cond:sup}
              c & = \sup_{ (f,y_0) \in \mathcal{F} \times \mathcal{V}} \limsup_{t\to\infty}\rho\left({\partial g[f](y_{g[f],y_0}(t),t)}/{\partial y} \right) \\
              & = \sup_{(f,y_0)\in \mathcal{L}_g} \lim_{t\to \infty} \rho\left({\partial g[f](y_{g[f],y_0}(t),t)}/{\partial y} \right)
       \end{align}
       holds for some $c >0$.
\end{definition}
\begin{remark}\label{rem:ed}
    Note that condition~\eqref{cond:sup} is trivial if the second equality is replaced by $\ge$ (since $\mathcal{F}\times\mathcal{V} \supseteq \mathcal{L}_g $). In this sense, what we actually assume is that equality holds there.
    All of the practical examples (ODEs) in~\cref{sec:example} satisfy 
    $\mathcal{L}_g = \mathcal{F}\times\mathcal{V}$, and the present authors are not certain whether considering the definition in the more general setting (possibly $\mathcal{F}\times\mathcal{V} \nsubseteq \mathcal{L}_g$) has any practical relevance.
    We nevertheless adopt this definition in order to keep the assumptions as weak as possible so that the subsequent discussion remains valid.
\end{remark}

Next, we define the \emph{essential convergence rate} based on $c$-essential dynamics. 
Let $\Phi[f, y](t)$ denote a metric for convergence, such as $f(y(t)) - f^\star$, $\norm{x(t) - x^\star}$ or $\norm{\nabla f(y(t))}$. When the ODE is extended to a $d'$-dimensional system, $x$ in $\norm{x(t) - x^\star}$ refers to the original solution vector $x_i = y_i$ ($i=1,\ldots,d$)).

As is well known, the convergence of optimization method can vary even for the same method, depending on objective functions and initial values.
Accordingly, it is common to evaluate the worst-case convergence rate over these variations.
In the same spirit, we take an analogous approach for continuous-time models: we define the essential convergence rate in terms of the worst-case rates.

\begin{definition}[Essential convergence rate] \label{def:ec}
       Suppose there exists a $c$-essential dynamics $g_c\in[g]$.
       We call the function $\beta_c $ \emph{the worst-case rate of $ g_c $ with respect to $\Phi$ in $\mathcal{F}\times \mathcal{V}$} if $ \Phi[f,y_{ g_c[f],y_0 }] (t) \le \beta_c (t) $ holds for all $ (f,y_0) \in \mathcal{F} \times \mathcal{V} $ and sufficiently large $ t \in \RR_{>0} $.        
       Then {\em the essential convergence rate} of $\Phi$ for $[g]$ 
       is defined by $\beta_c(t/c)$.
\end{definition}

Note that the essential convergence rate is not unique, since in \cref{def:ec}, any upper bound of $ \Phi[f,y_{ g[f],y_0 }] (t) $ can serve as an essential convergence rates.
This non-uniqueness is inevitable, because in continuous-time models, no rate lower-bounds are currently known, and thus there is no way to distinguish a {\em tight} rate.
In this context, our claim is that the essential rates are unaffected by time-rescaling, regardless of whether they are tight or not.
This parallels the situation in discrete-time optimization methods, where any upper bound is referred to as the convergence rate, whether or not it is tight.

The following proposition addresses potential concerns regarding the definition of the essential convergence rate: (i) even when there are multiple $c$-essential dynamics in $[g]$, the essential convergence rate $\beta_c(t/c)$ is well-defined; and (ii) although there is some freedom in the choice of $c$, this does not essentially affect the essential convergence rate.

\begin{proposition}\label{prop:welldef}
    Let $g_i$ $(i=1,2)$ be $c_i$-essential dynamics in $[g]$, and $y_{g_i[f],y_0}$ be associated solutions.
    Let $\beta_{c_i}(t/c_i)$ $(i=1,2)$ be their essential convergence rates in the sense of~\cref{def:ec}.
    Then, for any $a<1$, $\beta_{c_1}(t/c_1) = \Order{\beta_{c_2}(at/c_2)}$ and  $\beta_{c_2}(t/c_2) = \Order{\beta_{c_1}(at/c_1)}$ hold.
\end{proposition}

\begin{proof}
    Note that $y_{g_2[f],y_0}(t)$ can be expressed as $y_{g_1[f],y_0}(\alpha(t))$ for some time-rescaling function $\alpha(t)$.
    Since $g_i$ is $c_i$-essential, we have
    \begin{align}
       c_2 &= \sup_{ (f,y_0) \in \mathcal{F} \times \mathcal{V}} \limsup_{t\to\infty}\rho\left({\partial g_2[f](y_{g_2[f],y_0}(t),t)}/{\partial y} \right)\\
       &= \sup_{ (f,y_0) \in \mathcal{F} \times \mathcal{V} } \limsup_{t\to\infty} \dot{\alpha}(t) \rho\left({\partial g_1[f](y_{g_1[f],y_0}(\alpha(t)),\alpha(t))}/{\partial y} \right)\\
       &\ge \sup_{(f,y_0) \in \mathcal{L}_{g_1}} \limsup_{t\to\infty} \dot{\alpha}(t) \rho\left({\partial g_1[f](y_{g_1[f],y_0}(\alpha(t)),\alpha(t))}/{\partial y} \right)\\
       &= \limsup_{t\to\infty} \dot{\alpha}(t) \sup_{(f,y_0) \in \mathcal{L}_{g_1}} \lim_{t\to\infty} \rho\left({\partial g_1[f](y_{g_1[f],y_0}(t),t)}/{\partial y} \right)\\
       &= \limsup_{t\to\infty} \dot{\alpha}(t) c_1,
    \end{align}
    where the fourth line follows from the following equality: for real-value functions $u(t)$ and $v(t)$, if $\lim_{t\to\infty} v(t) = c > 0$, it holds that
    \begin{equation}
        \limsup_{t\to\infty} u(t)v(t) = (\limsup_{t\to \infty} u(t))c,
    \end{equation}
    allowing for both sides being $\infty$. 
    
    Similarly,
    \begin{align}
       c_1 &= \sup_{ (f,y_0) \in \mathcal{F} \times \mathcal{V}} \limsup_{\alpha(t)\to\infty}\rho\left({\partial g_1[f](y_{g_1[f],y_0}(\alpha(t)),\alpha(t))}/{\partial y} \right)\\
       &= \sup_{ (f,y_0) \in \mathcal{F} \times \mathcal{V}} \limsup_{t\to\infty} \frac{1}{\dot{\alpha}(t)} \rho\left({\partial g_2[f](y_{g_2[f],y_0}(t),t)}/{\partial y} \right)\\
       &\ge \sup_{(f,y_0) \in \mathcal{L}_{g_2}} \limsup_{t\to\infty} \frac{1}{\dot{\alpha}(t)} \rho\left({\partial g_2[f](y_{g_2[f],y_0}(t),t)}/{\partial y} \right)\\
       &= \limsup_{t\to\infty} \frac{1}{\dot{\alpha}(t)} \sup_{(f,y_0) \in \mathcal{L}_{g_2}} \lim_{t\to\infty} \rho\left({\partial g_2[f](y_{g_2[f],y_0}(t),t)}/{\partial y} \right)\\
       &= \frac{1}{\liminf_{t\to\infty} \dot{\alpha}(t)} c_2
    \end{align}
    holds. 
    Therefore, we have $c_2/c_1 \ge \limsup_{t\to\infty} \dot{\alpha}(t) \ge \liminf_{t\to\infty} \dot{\alpha}(t) \ge c_2/c_1$,
    which implies $\lim_{t\to\infty} \dot{\alpha}(t) = c_2/c_1$.
    Hence, for any $a <1$, $a\frac{c_2}{c_1}t \le \alpha(t)$ holds for sufficiently large $t$.
    Since $\beta_{c_2}(t) = \beta_{c_1}(\alpha(t))$, and since $\beta_{c_1}$ is decreasing, it follows that
    \begin{equation}
        \frac{\beta_{c_2}(t/c_2)}{\beta_{c_1}(at/c_1)} = \frac{\beta_{c_1}(\alpha(t/c_2))}{\beta_{c_1}(at/c_1)} = \Order{1}.
    \end{equation}
    The other inequality is shown in the same manner.
\end{proof}
\begin{remark}
       By the proposition, if the essential convergence rate is (inverse) polynomial $\Theta(1/t^p)$, then $p$ is uniquely determined.
       If the rate is exponential $\Theta(\e^{-q t})$, then $q$ is also unique.
       However, since $a < 1$ in the theorem, certain cases cannot be distinguished, although the indistinguishable differences are small relative to the essential convergence rates. For example, both $\e^{-t}$ and $\e^{-t}/t$ can be essential convergence rates for the same ODE simultaneously (such rates appear, for instance, in~\cite{LSY24}).
       
       Unfortunately, the condition $a < 1$ cannot be removed, i.e., we cannot claim $\beta_{c_2}(t/c_2) = \Order{\beta_{c_1}(t/c_1)}$ in \cref{prop:welldef}. 
       For example, $ \alpha (t) = (c_2/c_1) ( t - \log (t+1))$ is a time-rescaling function with $ \lim_{ t \to \infty} \dot{\alpha} (t) = c_2 / c_1 $. If we consider the case $ \beta_{c_1} (t) = \e^{-c_1 t }$ and $ \beta_{c_2} (t) = \beta_{c_1} (\alpha(t))$ (see the initial sentence of the proof), then 
       \[ \frac{\beta_{c_2}(t/c_2)}{\beta_{c_1}(t/c_1)} = \frac{\beta_{c_1}(\alpha(t/c_2))}{\e^{-t}} = \e^{ c_2 \log ( t/c_2 + 1) } = \paren*{ \frac{t}{c_2} + 1 }^{c_2} \to \infty, \]
       which implies $ \beta_{c_2} (t/c_2) \neq \Order{ \beta_{c_1} (t/c_1) }. $
\end{remark}


Now we show that convergence rates essentially cannot exceed the essential one in \cref{def:ec} by time-rescaling if discretization is taken into account. 

The setting for the theorem is as follows.
We consider the ODE~\eqref{ode4} for a dynamics $g\in \mathcal{G}$, with the objective function $f$ varying in $\mathcal{F}$.
The step sizes are denoted by $h_k$, and the accumulated time by $t_k \coloneqq \sum_{i=1}^{k} h_k$ (with $ t_0 = 0 $). 
We denote the stability domain of a given fixed numerical method by $S$, and write its radius as $r:=\max_{z\in S}|z|$.

\begin{theorem}\label{thm}
       Assume the following.

       (Assumption on the dynamics) There exists a $1$-essential dynamics $g_1$ in $[g]$.
       We then consider the rescaled ODE~\eqref{ode4:trans}, and write the right-hand side as $\hat{g}[f](y,t) \coloneqq \dot{\alpha}(t)g_1[f](y,\alpha(t))$ with a time-rescaling function $ \alpha $.
        Its solution, with the initial data $y_0$  varying in $\mathcal{V}$, is again denoted by $y_{\hat{g}[f],y_0}(t)$.

       (Assumption on the time-rescaling function) $\dot{\alpha}$ is monotonic.

       (Assumptions on the applied numerical method) 
       We use the same numerical method throughout the integration.
       We denote the associated numerical solution by $\setE{y_{\hat{g}[f],y_0}^{(k)}}_k$.
       The method and its step-size control satisfy the following.
       \begin{enumerate}[label=\rmfamily(\roman*)]
              \item The stability domain $S$ of the numerical method is bounded (i.e., $r< + \infty$) and static (i.e., it does not change with time).
              \label{asm1}
              \item For any $f \in \mathcal{F}$ and $y_0 \in \mathcal{V}$, 
              $h_k$ is controlled so that
              all the eigenvalues of $h_k \pdv{\hat{g}[f]}{y}\,\!(y_{\hat{g}[f],y_0}^{(k-1)},t_{k-1})$ lie in the stability domain, 
              and $ t_k \to \infty $ as $ k \to \infty $.
              \label{asm3}
              \item For any $\varepsilon_0 > 0$, there exist $\tilde{f} \in \mathcal{F}$, $\tilde{y}_0 \in \mathcal{V}$ and $k_0 >0$ such that for all $k \ge k_0$,
              \begin{equation}
                     \rho\paren*{\pdv{g_1[\tilde{f}]}{y}(y_{\hat{g}[\tilde{f}],\tilde{y}_0}^{(k)},\alpha(t_k))} \ge 1-\varepsilon_0.
              \end{equation}\label{asm5} 
       \end{enumerate}
       
       Then, for any $\varepsilon > 0$, there exists an infinite subsequence of time $\setE{t_{\ell_j}}_j$ such that
       \begin{equation}
              \alpha(t_{\ell_j}) \le r \ell_j + \varepsilon \ell_j 
       \end{equation}
       holds.
\end{theorem}

Before presenting the proof, we offer an intuitive interpretation of the theorem.
Note that $t_{\ell_j}$ represents the (discrete) elapsed time in the time scale of $\hat{g}$, whereas $\alpha(t_{\ell_j})$ denotes the elapsed time in the scale of $g_1$. 
The claim that the latter equals $r\ell_j $ (up to an arbitrarily small $\varepsilon \ell_j$) implies that, however fast the rate may appear in the ``(hopefully) accelerated'' ODE $\dot{y} = \hat{g}[f](y,t)$, it is actually no faster than a fixed step-size integration with step size $r$ of the $g_1$ ODE, the $1$-essential dynamics.
In this sense, \cref{thm} tells us that time-rescaling in continuous-time model is not a magic trick in this general setting---the best we can expect is the constant times speeding-up, namely $r$, which comes from the size of the stability domain $S$.


The linear speedup can be significant when the original convergence rate is of the form $\Order{\e^{-qt}}$ for some $q\in \RR_{>0}$.
In this case, the predicted rate $\Order{\e^{-rqk}}$ is strictly faster when $r>1$.
This is interesting, as it suggests a way to speed up the discrete convergence rate, genuinely in the discretization stage.
More precisely, if we can choose or design a good numerical method such that all the (possibly time-varying) eigenvalues of the Jacobian matrix lie along the direction of the radius of the stability domain, we can expect $r$-times acceleration.
However, it should be noted that for extremely large time steps $h_k>1$, the numerical method may no longer approximate the true solution well, and in such cases there is little hope that the discrete solutions will inherit the continuous convergence rate.

\begin{remark}
        Among the assumptions of \cref{thm}, condition~\ref{asm5} may appear technical, but it can be interpreted as requiring that the asymptotic behavior of the numerical solution is reasonable. 
        Indeed, condition~\ref{asm5} not only appears to be the discrete-time counterpart of \cref{def:ec}, but also follows from the assumption that $g_1$ is $1$-essential together with the following additional assumptions:
        for any $\varepsilon_0 > 0$, there exist $\tilde{f} \in \mathcal{F}$, $\tilde{y}_0 \in \mathcal{V}$ and $k_0 >0$ such that for all $k \ge k_0$,
       \begin{equation}
              \rho\paren*{\pdv{g_1[\tilde{f}]}{y}(y_{\hat{g}[\tilde{f}],\tilde{y}_0}(t_k),\alpha(t_k))} \ge 1 - \varepsilon_0, \label{asm5-1}
       \end{equation}
       and
       \begin{equation}
              \rho\paren*{\pdv{g_1[\tilde{f}]}{y}(y_{\hat{g}[\tilde{f}],\tilde{y}_0}(t_k),\alpha(t_k))} - \rho\paren*{\pdv{g_1[\tilde{f}]}{y}(y_{\hat{g}[\tilde{f}],\tilde{y}_0}^{(k)},\alpha(t_k))} \le \varepsilon_0. \label{asm5-2}
       \end{equation}
       
       The existence of $ (\tilde{f}, \tilde{y}_0, k_0) $ satisfying only \eqref{asm5-1} follows directly from the assumption that $g_1$ is $1$-essential, since
       \begin{align}
       1&= \sup_{ (f,y_0) \in \mathcal{F} \times \mathcal{V} } \limsup_{t \to \infty} \rho\paren*{\pdv{g_1[f]}{y}(y_{g_1[f],y_0}(\alpha(t)),\alpha(t))}\\
       &= \sup_{ (f,y_0) \in \mathcal{F} \times \mathcal{V} } \limsup_{t \to \infty} \rho\paren*{\pdv{g_1[f]}{y}(y_{\hat{g}[f],y_0}(t),\alpha(t))}.
       \end{align}
       
       It is nontrivial that the same triplet $ (\tilde{f}, \tilde{y}_0, k_0) $ also satisfies \eqref{asm5-2}, but this holds, for example, when the continuous- and discrete-time solutions converge to the same point. 
       Unfortunately, however, these two solutions may converge to different points when, for instance, there are multiple (local) optima.
       To cover such cases, we adopt the weaker assumption---one that is satisfied by a wide range of known ODE models, including all those presented in \cref{sec:example}.
\end{remark}

In the proof of~\cref{thm}, we use the following lemma.
\begin{lemma}[Silverman--Toeplitz theorem (cf.~\cite{H49b})]\label{lem}
       Define the transformation of sequences $\bm{\mu} \colon \setE{s_n}_{n=0}^\infty \mapsto \setE{t_n}_{n=0}^\infty$ by $t_n = \sum_{k=0}^{n} \mu_{nk} s_k$ where $\mu_{nk}\in\RR$. For any convergent sequence $\setE{s_n}$, the transformation $\bm{\mu}$ preserves the limit if and only if the following three conditions hold:
       \begin{enumerate}[label=\rmfamily(\alph*)]
              \item there exists $M > 0$ such that $\sum_{k=0}^{n} \abs{\mu_{nk}} \le M$ for all $n$,
              \item $\lim_{n \to \infty} \sum_{k=0}^{n} \mu_{nk} = 1$,
              \item $\lim_{n \to \infty} \mu_{nk} = 0$ for any fixed $k$.
       \end{enumerate}
\end{lemma}

\begin{proof}[Proof of Theorem~\ref{thm}]
       In the following, $\varepsilon$ is an arbitrary positive number.
       
       Let $\varepsilon_0 = 1 - (1 + \frac{\varepsilon}{3r})^{-1}$ and for this $\varepsilon_0$, take $(\tilde{f}, \tilde{y}_0, k_0) $ in the assumption~\ref{asm5}.
       Below we only consider iterations after $k_0$ which is enough to state the claim. From the assumptions~\ref{asm1} and~\ref{asm3}, for the stability for the ODE with $\tilde{f}$ and $\tilde{y}_0$, the time step size $h_k$ should satisfy
       \begin{equation}
              \left| h_k \dot{\alpha}(t_{k-1}) \rho \left( \pdv{g_1[\tilde{f}]}{y}(y_{\hat{g}[\tilde{f}],\tilde{y}_0}^{(k-1)},\alpha(t_{k-1})) \right) \right| \le r, 
       \end{equation}
       which implies
       \begin{equation}
              h_k \dot{\alpha}(t_{k-1}) \le \frac{r}{1-\varepsilon_0} = r + \frac{\varepsilon}{3} \qquad (k=k_0+1, k_0+2, \ldots). \label{b}
       \end{equation}

       With this observation on time step sizes, a rough sketch of the proof is immediate:
       \begin{align}
              \alpha(t_{k_0+k}) - \alpha(t_{k_0})
               = \sum_{i = k_0 + 1}^{k_0+ k} \int_{t_{i-1}}^{t_i} \dot{\alpha}(t) \dd t 
               \simeq  \sum_{i = k_0 + 1}^{k_0+k} h_i \dot{\alpha}(t_{i-1})
              \le  rk + \frac{\varepsilon}{3} k. \label{rough}
       \end{align}


When $\dot{\alpha}$ is weakly monotonically decreasing (recall the assumption (ii) on $\alpha(t)$), $\simeq$ can be replaced with $\le$, 
and the proof is complete. Thus, our main task is to establish a similar estimate on the other case.
       
       When $\dot{\alpha}$ is weakly monotonically increasing, we have instead
       \begin{align}
              \alpha(t_{k_0+k}) - \alpha(t_{k_0})
              = \sum_{i = k_0+1}^{k_0+k} \int_{t_{i-1}}^{t_i} \dot{\alpha}(t)\dd t 
              \le  \sum_{i =k_0+ 1}^{k_0+k} h_i \dot{\alpha}(t_{i})
              \le  \sum_{i = k_0+1}^{k_0+k} \paren*{r + \frac{\varepsilon}{3}} \frac{\dot{\alpha}(t_i)}{\dot{\alpha}(t_{i-1})}. \label{inc}
       \end{align}


       Let us consider the sequence $E_k \coloneqq \dot{\alpha}(t_k)/\dot{\alpha}(t_{k-1})$ $(k=k_0+1, k_0+2,\ldots)$ and extract a subsequence of indices $\setE{\ell_j}_j$ corresponding to large elements of $\setE{E_{k}}_k$: 
       \begin{equation}
           \setE{ \ell_j }_{j=0}^{\infty} \coloneqq \{ k_0 \} \cup \setI*{ k }{ E_k > \frac{r + \frac{2\varepsilon}{3}}{r + \frac{\varepsilon}{3}} \eqqcolon \gamma\, (>1)}. 
       \end{equation}
       Here, we include $ \ell_0 = k_0 $ to simplify the discussion although $ E_{\ell_0 } $ is not defined above. 
       
       If $\setE{E_{\ell_j}}$ is a finite sequence, the claim is obvious.
       Otherwise, let us introduce $J_k$ as the largest index $j$ such that $\ell_{j} \le k_0 + k$.
       With these notation and by splitting the sum in the last term of~\eqref{inc} into the large elements and the rest, we see
       \begin{align}
            \MoveEqLeft
              \alpha(t_{k_0+ k}) - \alpha(t_{k_0})\\
              & \le \paren*{r + \frac{\varepsilon}{3}}\frac{r + \frac{2\varepsilon}{3}}{r + \frac{\varepsilon}{3}}k + \sum_{j=1}^{J_k} \paren*{r + \frac{\varepsilon}{3}} E_{\ell_j} = rk + \frac{2\varepsilon}{3}k+ \sum_{j=1}^{J_k} \paren*{r + \frac{\varepsilon}{3}} E_{\ell_j}.
       \end{align}

       We establish the claim by showing $\sum_{j=1}^{J_k} \paren*{r + \frac{\varepsilon}{3}} E_{\ell_j} \le \frac{\varepsilon}{3}k$ holds for infinitely many $k$. Consider the subsequence of time $\setE{t_{\ell_j}}_j \subseteq \setE{t_k}_k$. It is sufficient to prove that
       \begin{equation}
              \frac{\sum_{j=1}^{J} E_{\ell_j}}{\ell_{J+1}-\ell_1}
              = \sum_{j=1}^{J} \frac{\ell_{j+1} - \ell_j}{\sum_{s=1}^{J} (\ell_{s+1} - \ell_s)}  \paren*{\frac{E_{\ell_j}}{\ell_{j+1} - \ell_j}} \to 0 \quad \text{as} \quad J \to \infty.
       \end{equation}
       Letting $\mu_{J_k j} = (\ell_{j+1} - \ell_j)/\sum_{l=1}^{J_k} (\ell_{l+1} - \ell_l)$, we observe that $\setE{\mu_{J_k j}}$ satisfies the conditions (a), (b), and (c) in~\cref{lem}.
       Thus,
       \begin{equation}
              \lim_{J \to \infty} \frac{\sum_{j=1}^{J} E_{\ell_j}}{\ell_{J+1}-\ell_1} = \lim_{j \to \infty} \frac{E_{\ell_j}}{\ell_{j+1} - \ell_j}.
       \end{equation}

       Therefore, we complete the proof by showing ${E_{\ell_j}}/(\ell_{j+1} - \ell_j) \to 0$ as $j \to \infty$.
       Since $\dot{\alpha}$ is weakly monotonically increasing, and since $E_{k} \ge 1$ for $k \ge k_0$ and $E_{\ell_j} > \gamma$ for $j \ge 1$,
       \begin{align}
              \sum_{i=k_0+1}^{k_0+k} \frac{1}{\dot{\alpha}(t_{i-1})}
              &\le \sum_{j = 0}^{J_k} \frac{1}{\dot{\alpha}(t_{\ell_j})} (\ell_{j+1} - \ell_j)\label{J}\\
              &= \frac{ \ell_{1} - \ell_0 }{\dot{\alpha}(t_{\ell_0})} + \sum_{j = 1}^{J_k} \left[ \frac{1}{\dot{\alpha}(t_{\ell_0})} \paren*{\prod_{i=k_0+1}^{\ell_j}\frac{1}{E(t_i)}} (\ell_{j+1} - \ell_j) \right]\\
              &< \frac{ \ell_{1} - \ell_0 }{\dot{\alpha}(t_{\ell_0})} + \sum_{j = 1}^{J_k} \frac{1}{\dot{\alpha}(t_{\ell_0})} \frac{1}{\gamma^{j-1}}\frac{1}{E_{\ell_{j}}} (\ell_{j+1} - \ell_j). 
       \end{align}
       Here if we take the limit of $k\to\infty$, the most left hand side should tend to $\infty$, which is visible from the simple observation
       \begin{align}
              \lim_{k \to \infty} (t_{k_0+k} - t_{k_0}) 
              = \lim_{k \to \infty}\sum_{i=k_0+1}^{k_0+k} h_i 
              \le \lim_{k \to \infty} \paren*{r + \frac{\varepsilon}{3}} \sum_{i=k_0+1}^{k_0+k} \frac{1}{\dot{\alpha}(t_{i-1})}.
       \end{align}
       Since $t_k\to\infty$ as $k\to\infty$ from the assumption~(\ref{asm3}), $\sum 1/\dot{\alpha}$ should be so as well.
       Thus,~\eqref{J} reveals
       \begin{equation}
              \sum_{j = 1}^{\infty} \frac{1}{\gamma^{j-1}}\frac{1}{E_{\ell_{j}}} (\ell_{j+1} - \ell_j) = \infty,
       \end{equation}
       which implies
       \begin{equation}
              \frac{1}{\gamma^{j-1}}\frac{1}{E_{\ell_{j}}} (\ell_{j+1} - \ell_j) = \Omega\left( \frac{1}{j^2} \right).
       \end{equation}
       Therefore, since $\gamma > 1$, we have ${E(t_{\ell_{j}})}/(\ell_{j+1} - \ell_j) \to 0$ as $j \to \infty$.
\end{proof}

\section{Examples}\label{sec:example}
In this section, we present several examples of essential convergence rates and compare them.
In the examples below, we focus on two types of ODE systems: the gradient flow, and the systems of the form \[
\simulparen{
\dot{x} &= g_1[f](x,v,t)\\
\dot{v} &= g_2[f](x,v,t).
}
\]
Using the notation introduced in \cref{sec:preliminary}, we can write $y = (x^\top, v^\top)^\top$ and $\dot{y} = g[f](y,t) = (g_1[f](y,t)^\top, g_2[f](y,t)^\top)^\top$.
In this section, however, we prefer the notation $(x,v)$ since that is more common in the optimization literature.

We consider a time-rescaling for each ODE such that it becomes $1$-essential, i.e., it satisfies condition~\eqref{cond:sup} with $c = 1$.
For verifying this condition, it is convenient to check the following two inequalities:
\begin{align}
    \sup_{ (f,y_0) \in \mathcal{F} \times \mathcal{V}} \limsup_{t\to\infty}\rho\left({\partial g[f](y_{g[f],y_0}(t),t)}/{\partial y} \right) &\le 1, \label{cond:ed:ub} \\
    \sup_{(f,y_0)\in \mathcal{L}_g} \lim_{t\to \infty} \rho\left({\partial g[f](y_{g[f],y_0}(t),t)}/{\partial y} \right) &\ge 1. \label{cond:ed:lb}
\end{align}
Together, these conditions immediately yield the desired claim, in view of the general relation, 
\[ \sup_{ (f,y_0) \in \mathcal{F} \times \mathcal{V}} \limsup_{t\to\infty}\rho\left(\frac{\partial g[f](y_{g[f],y_0}(t),t)}{\partial y} \right) \ge \sup_{(f,y_0)\in \mathcal{L}_g} \lim_{t\to \infty} \rho\left(\frac{\partial g[f](y_{g[f],y_0}(t),t)}{\partial y} \right) \]
as already noted in \cref{rem:ed}.

In this section, we restrict attention to $\mathcal{F} = \mathcal{S}_{\mu,L}$, where $\mu$ may be zero.
For $f \in \mathcal{S}_{\mu,L}$, let $(L \ge)\, \lambda_1 \ge \ldots \ge \lambda_d \,(\ge \mu)$ denote the eigenvalues of $\nabla^2 f(x)$.
Although these eigenvalues depend on the solution, we simply write $ \lambda_1,\dots,\lambda_d $ to avoid cumbersome notation.

Given a second-order ODE, the procedure for transforming it into a first-order system $\dot{y} = g[f](y,t)$ is not unique.
However, if we restrict ourselves to transformations satisfying certain conditions,
it can be shown that the eigenvalues of $g[f](y,t)$ remain invariant under such transformations (see~\cref{app:firstorder}).

\subsection{Gradient flow}
For $f \in \mathcal{S}_{\mu,L}$, consider the (rescaled) gradient flow
\begin{equation}\label{eq:ode:gf}
       \dot{x} = - \dot{\alpha} (t) \nabla f(x), \quad x(0) = x_0 \in \RR^d.
\end{equation}
The convergence rate is given by $f(x(t)) - f^\star \le \frac{1}{2 \alpha(t) }{\norm{x_0-x^\star}^2}$ if $\mu = 0$~\cite{SBC16} and $f(x(t)) - f^\star \le \e^{-2\mu \alpha(t) }(f(x_0) - f^\star)$ if $\mu > 0$~\cite{SRBd17}.
The set $ \mathcal{V} $ of initial values is defined as $ \mathcal{V} = \setI{x_0 \in \RR^d }{\norm{x_0 - x^\star} \le R}$ if $ \mu = 0 $ and $ \mathcal{V} = \setI{x_0 \in \RR^d }{f(x_0) - f^\star \le R}$ if $ \mu > 0 $. 
As this example illustrates, the setting of $ \mathcal{V}$ includes the specification of ``quantities regarded as constant in the convergence rates''.

We see that the ODE becomes 1-essential when $\alpha(t) = t/L$.
Indeed, the eigenvalue of the Jacobian of \eqref{eq:ode:gf} are $ - \dot{\alpha} (t) \lambda_i $, whose absolute values are bounded above by $\dot{\alpha} (t)L$.
Thus, when $ \alpha (t) = t / L $, condition~\cref{cond:ed:ub} is satisfied.
Moreover, by taking $f \in \mathcal{S}_{\mu,L}$ such that $f(x) = \frac{L}{2} \norm{x}^2$, we see that condition~\eqref{cond:ed:lb} also holds.

Hence, in view of the convergence rates of the gradient flow given above, the essential convergence rates of $f(x) - f^\star$ are $\Order{{1}/{t}}$ when $\mu = 0$ and $\Order{\e^{-2(\mu/L) t}}$ when $\mu > 0$.

\subsection{ODE for accelerated gradient method}\label{subsec:NAG}
For $f \in \mathcal{S}_{0,L}$, consider the following rescaled ODE:
\begin{equation}
       \simulparen{
              \dot{x} &= \frac{\dot{\alpha}(t)}{\alpha(t)}(v-x)\\
              \dot{v} &= -\frac{\dot{\alpha}(t)}{4}\nabla f(x),
       }\label{NAG}
\end{equation}
where $(x(0),v(0)) \in \mathcal{V} = \setI{(x_0,x_0) \in \RR^d \times \RR^d}{\norm{x_0 - x^\star} \le R}$.
If $\alpha(t) = t^2$, the ODE coincides with the continuous-time limit of AGM~\eqref{example}, and if $\alpha(t) = \sqrt{2}t^2$, it coincides with that of the optimized gradient method (OGM)~\cite{DT14,KF16}.
For the solution $(x(t), v(t))$ of  ODE~\eqref{NAG}, it holds that $f(x(t)) - f^\star \le {2\norm{x_0 - x^\star}^2}/{\alpha(t)}$~\cite{WRJ21}.

In the following, we derive that when $\alpha(t) = t^2/L$, 
the rescaled ODE is $1$-essential.
The eigenvalues of the Jacobian of ODE~\eqref{NAG} are given by
\begin{equation}
       \frac{\dot{\alpha}(t)}{2 \alpha(t)} \paren*{-1 \pm \sqrt{ 1 - \alpha (t) \lambda_i }} \quad (i=1,\ldots,d).
\end{equation}
For indices $i$ such that $ \lambda_i = 0 $, the corresponding eigenvalues are $ - \dot{\alpha}(t) / \alpha (t) $ and $0$. 
On the other hand, for indices $i$ with $ \lambda_i > 0 $ and sufficiently large $t$, the corresponding eigenvalues are complex number, with absolute values $ \dot{\alpha}(t)/2 \cdot \sqrt{\lambda_i / \alpha(t)} $. 
Consequently, the spectral radius satisfies
\[ \rho\left(\frac{\partial g[f](y_{g[f],y_0}(t),t)}{\partial y} \right) \le \max \setE*{ \frac{\dot{\alpha}(t)}{\alpha(t)}, \frac{\dot{\alpha}(t)}{2} \sqrt{ \frac{L}{ \alpha(t) } } }. \]
Therefore, when $\alpha(t) = t^2/L$, the condition~\eqref{cond:ed:ub} is satisfied. 
Moreover, by taking $f \in \mathcal{S}_{0,L}$ such that $f(x) = \frac{L}{2} \norm{x}^2$, we see the condition~\eqref{cond:ed:lb} is also satisfied.
This ODE is equivalent to 
\begin{equation}
       \ddot{x} + \frac{3}{t} \dot{x} + \frac{1}{L}\nabla f(x) = 0. \label{NAGLODE}
\end{equation}

Hence, the essential convergence rate of $f(x) - f^\star$ is ${2L\norm{x_0-x^\star}^2}/{t^2}$.
In comparison, the convergence rate of AGM is $f(x_k) - f^\star \le {2L\norm{x_0-x^\star}^2}/{k^2}$~\cite{N83}, while that of OGM is bounded by $f(x_k) - f^\star \le {L\norm{x_0-x^\star}^2}/{k^2}$~\cite{KF16}.
This implies that although the underlining ODE is the same, the discretization of OGM yields a faster convergence rate than AGM by a factor of $\sqrt{2}$.

\subsection{ODE for accelerated gradient method with shifted gradient}
For $f \in \mathcal{S}_{0,L}$ and $b >0$, consider the following ODE:
\begin{equation}
       \simulparen{
              \dot{x} &= \frac{2 \dot{\alpha} (t)}{\alpha(t)}(v-x)\\
              \dot{v} &= -\frac{\alpha(t) \dot{\alpha} (t)}{2}\nabla f\paren*{\paren*{1-\frac{2b}{\alpha(t)}}x + \frac{2b}{\alpha(t)}v},
       }\label{NAGS}
\end{equation}
where $(x(0),v(0)) \in \mathcal{V} = \setI{(x_0,x_0) \in \RR^d \times \RR^d}{\norm{x_0 - x^\star} \le R}$. 
If $\alpha(t) = t$, this ODE coincides with~\eqref{ODEshifted}. 
For ODE~\eqref{NAGS}, it holds that
\begin{equation}
       \min_{0\le s \le t} \norm*{\nabla f \paren*{x(s) + \frac{b}{\dot{\alpha}(s)}\dot{x}(s)}}^2 \le \frac{18\norm{x_0-x^\star}}{7b (\alpha(t))^3}
\end{equation}
(see \eqref{eq:rate:NAGS}). 

The appropriate time scaling of ODE~\eqref{NAGS} depends on $b$. 
The eigenvalues of its Jacobian are
\[ \dot{\alpha} (t) \paren*{ - \paren*{ \frac{1}{\alpha(t)} + \frac{b \lambda_i}{2} } \pm \sqrt{ \paren*{ \frac{1}{\alpha(t)} + \frac{b \lambda_i}{2} }^2 - \lambda_i } }. \]
Since 
\[ \lim_{ t \to \infty } \paren*{ - \paren*{ \frac{1}{\alpha(t)} + \frac{b \lambda_i}{2} } \pm \sqrt{ \paren*{ \frac{1}{\alpha(t)} + \frac{b \lambda_i}{2} }^2 - \lambda_i } } = - \frac{b \lambda_i }{2} \pm \sqrt{ \frac{b^2 \lambda_i^2}{4} - \lambda_i }, \]
which does not vanish for nonzero $ \lambda_i $, 
it is necessary that $ \dot{\alpha} (t) = \Order{1} $ in order to ensure condition~\eqref{cond:ed:lb}. 
Therefore, we assume $ \lim_{ t \to \infty } \dot{\alpha} (t) = a < \infty $. 
Then, dividing into cases according to the sign of $ \frac{b^2 L^2}{4} - L $, we obtain
\begin{align}
    \limsup_{ t \to \infty } \rho\left(\frac{\partial g[f](y_{g[f],y_0}(t),t)}{\partial y} \right)
    &\le a \times \max \setE*{ \sqrt{L}, \max_{ \lambda \in [0,L] \cap \left[\frac{4}{b^2},\infty \right) } \paren*{ \frac{b \lambda}{2} + \sqrt{ \frac{b^2 \lambda^2}{4} - \lambda } } }\\
    &= \begin{cases} a \sqrt{L} & \paren*{ b \le \frac{2}{\sqrt{L}} } \\ \frac{a}{2} \paren*{ bL + \sqrt{ (bL)^2 - 4 L } } & \paren*{ b > \frac{2}{\sqrt{L}} } \end{cases}. 
\end{align}

When $b \le 2/\sqrt{L}$, choosing $\alpha(t) = t/\sqrt{L}$ makes the ODE $1$-essential. 
The above discussion verifies the condition~\eqref{cond:ed:ub} and condition~\eqref{cond:ed:lb} can be confirmed by taking $f(x) = \frac{L}{2} \norm{x}^2$ as well as the previous example. 
Then, ${18L\sqrt{L}\norm{x_0-x^\star}^2}/{(7bt^3)}$ is an essential convergence rate of $\min_{0\le s \le t} \norm{\nabla f(x(s) + b\sqrt{L} \dot{x}(s))}^2$. 

Similarly, when $b > 2/\sqrt{L}$, choosing $\alpha(t) = {2t}/{(bL + \sqrt{(bL)^2 - 4L})}$ makes the ODE $1$-essential, and 
${9 \paren*{ bL + \sqrt{(bL)^2 - 4L} }^3 \norm{x_0-x^\star}^2}/{(28bt^3)}$ is an essential convergence rate. 

Consequently, the optimal choice of $b$ is $b = 2/\sqrt{L}$, which yields the optimal essential convergence rate ${9L^2\norm{x_0-x^\star}^2}/{(7t^3)}$.
Given its connection with the second-order form \eqref{ODEshifted}, the system \eqref{NAGS} with this choice of $\alpha$ and $b$ can be regarded as an optimal implicit-velocity ODE~\cite{CSY22}.

\subsection{ODE for accelerated gradient method for strongly convex functions}
For $f \in \mathcal{S}_{\mu,L}$ where $\mu > 0$,
consider the following ODE:
\begin{equation}
       \simulparen{
              \dot{x} &= \dot{\alpha}(t) \sqrt{\mu}(v-x)\\
              \dot{v} &= \dot{\alpha}(t) \sqrt{\mu}(x-v) - \frac{\dot{\alpha}(t)}{\sqrt{\mu}} \nabla f(x),\label{agmsc}
       }
\end{equation}
where $(x(0),v(0)) \in \mathcal{V} = \setI{(x_0,x_0) \in \RR^d \times \RR^d}{f(x_0) - f^\star + \frac{\mu}{2}\norm{x_0 - x^\star}^2 \le R}$. 
This is a continuous-time limit of AGM for strongly convex functions, with convergence rate $f(x(t)) - f^\star \le \e^{-\sqrt{\mu} \alpha (t)}(f(x_0) - f^\star + \frac{\mu}{2}\norm{x_0 - x^\star}^2)$~\cite{WRJ21}.

When $\alpha(t) = t/\sqrt{L}$, the ODE is $1$-essential.
Indeed, the eigenvalues of the Jacobian matrix are
\begin{equation}
       \dot{\alpha} (t) \paren*{ -\sqrt{\mu} \pm \sqrt{\lambda_i-\mu}\,\mathrm{i} } \quad (i=1,\ldots,d) 
\end{equation}
and hence the spectral radius is bounded by $ \dot{\alpha} (t) \sqrt{L} $. 
Therefore, condition~\eqref{cond:ed:ub} is satisfied when $ \alpha(t) = t / \sqrt{L}$. 
In addition, by taking $f(x) = \frac{L}{2} \norm{x}^2$, condition~\eqref{cond:ed:lb} is also satisfied.
In this case, the ODE becomes
\begin{equation}
       \ddot{x} + 2\sqrt{\frac{\mu}{L}} \dot{x} + \frac{1}{L}\nabla f(x) = 0.
\end{equation}

As a consequence, $\Order{\e^{-\sqrt{\mu/L} t}}$ is an essential convergence rate of $f(x(t)) - f^\star$. 
This perfectly matches the convergence rate of AGM for strongly convex functions $f(x_k) - f^\star \le (1 - \sqrt{\mu/L})^k(f(x_0) - f^\star + \frac{\mu}{2}\norm{x_0 - x^\star}^2)$.

\subsection{ODE for Triple Momentum Method (TMM)}
For $f \in \mathcal{S}_{\mu,L}$ where $\mu > 0$,
consider the following ODE:
\begin{equation}
       \simulparen{
              \dot{x} &= 2\dot{\alpha}(t)\sqrt{\mu}(v-x)\\
              \dot{v} &= \dot{\alpha}(t)\sqrt{\mu}(x-v) - \frac{\dot{\alpha}(t)}{\sqrt{\mu}}\nabla f(x),
       }
\end{equation}
where $(x(0),v(0)) \in \mathcal{V} = \setI{(x_0,x_0) \in \RR^d \times \RR^d}{f(x_0) - f^\star + \frac{\mu}{2}\norm{x_0 - x^\star}^2 \le R}$.
This is a continuous-time limit of TMM, with convergence rate $\norm{x(t) + \frac{1}{2\dot{\alpha}(t)\sqrt{\mu}}\dot{x}(t) - x^\star}^2 \le \e^{-2\sqrt{\mu} \alpha(t)}(\frac{1}{\mu}(f(x_0) - f^\star) + \frac{1}{2}\norm{x_0 - x^\star}^2)$~\cite{USM24}.

When $\alpha(t) = t/\max \setE{2\sqrt{\mu}, \sqrt{2L}}$, the ODE becomes $1$-essential.
Indeed, the eigenvalues of the Jacobian matrix are
\begin{equation}
       \frac{\dot{\alpha} (t)}{2} \paren*{ -3\sqrt{\mu} \pm \sqrt{9\mu-8\lambda_i}} \quad (i=1,\ldots,d).
\end{equation}
Therefore, the spectral radius can be evaluated as 
\begin{align}
\limsup_{ t \to \infty } \rho\left(\frac{\partial g[f](y_{g[f],y_0}(t),t)}{\partial y} \right) 
&\le \lim_{ t \to \infty } \paren*{ \frac{\dot{\alpha} (t)}{2} \max_{ \lambda \in [\mu, L] } \abs*{ -3\sqrt{\mu} \pm \sqrt{9\mu-8\lambda} } } \\
&= \paren*{ \lim_{ t \to \infty } \dot{\alpha} (t) } \cdot \max \setE*{ 2 \sqrt{\mu}, \sqrt{2L} },
\end{align}
which implies that condition~\eqref{cond:ed:ub} holds when $ \alpha (t) = t/\max \setE{2\sqrt{\mu}, \sqrt{2L}} $. 
In addition, by taking $f \in \mathcal{S}_{\mu,L}$ such that
\begin{equation}
       f(x) = \begin{cases}
              \frac{L}{2} \norm{x}^2 & \text{if } \sqrt{2\mu} \le \sqrt{L},\\
              \frac{\mu}{2} \norm{x}^2 & \text{otherwise},
       \end{cases}
\end{equation}
we can confirm condition~\eqref{cond:ed:lb} is satisfied.

When $\sqrt{L} \ge \sqrt{2\mu}$, i.e., $\alpha(t) = t/\sqrt{2L}$, the ODE becomes
\begin{equation}
    \ddot{x} + \sqrt{ \frac{2\mu}{L} } \dot{x} + \frac{1}{2L} \nabla f(x) = 0,
\end{equation}
and an essential convergence rate of $\norm*{x(t) + \sqrt{L/(2\mu)}\dot{x}(t) - x^\star}^2$ is $\Order{\e^{-\sqrt{2\mu/L} t}}$,
which is faster than that of AGM ODE~\eqref{agmsc} by a factor of $\sqrt{2}$.
However, the discrete convergence rate of TMM is $\Order{(1-2\sqrt{\mu/L})^k}$, which is faster than the essential convergence rate of the TMM ODE by a factor of $\sqrt{2}$.
This implies that the twice faster discrete convergence rate of TMM than AGM is achieved through a combination of the underlining ODE and its discretization.



\section{Discussion}\label{sec:discussion}
In this section, we discuss the consequence and assumptions of \cref{thm}.

\subsection{Recovering rates from decelerated ODEs}
From \cref{thm}, we can draw two conclusions.

The first, as previously mentioned, is that no matter how much the ODE is accelerated, stable discretization inevitably slows it down, resulting in the essential convergence rate.

The second conclusion is that, starting from a decelerated ODE, it is possible to accelerate the convergence rate upon discretization by increasing the step sizes. However, the resulting rate cannot exceed the essential convergence rate. This naturally raises the question: can the resulting rate always match the essential convergence rate?

The answer is negative. The following example shows that the essential convergence rate cannot be recovered if the step sizes are restricted to satisfy the stability conditions derived from linear stability analysis.

Let $\mathcal{F} = \setI{f\colon \RR \to \RR}{f(x) = x^4 / 4}$, $\mathcal{V} = \setE{1}$, and consider the gradient flow
\begin{equation}
       \dot{x} = -\nabla f(x) = -x^3, \quad x(0) = 1. \label{x3}
\end{equation}
The solution is
\begin{equation}
       x(t) = \frac{1}{\sqrt{2t + 1}},
\end{equation}
and therefore the convergence rate is $f(x(t)) - f^\star = \Theta(1/t^2)$.

However, the Jacobian of the right-hand side of~\eqref{x3} is
\begin{equation}
       \pdv{}{x}(-x^3)(t) = \frac{-3}{2t + 1} \to 0 \quad \text{as} \quad t \to \infty, \label{nlin}
\end{equation}
which implies that ODE~\eqref{x3} is not 1-essential.

Now consider the time-rescaled gradient flow
\begin{equation}
       \dot{x} = -\dot{\alpha}(t) x^3, \quad x(0) = 1. \label{x31}
\end{equation}
When $\alpha(t) = \frac{3}{2}(\e^{\frac{2}{3}t}-1)$, it becomes 1-essential since its solution is
\[
    x(t) = \frac{1}{\sqrt{3(\e^{\frac{2}{3}t}-1)+1}},
\]
and the Jacobian of the right-hand side of \eqref{x31} is
\begin{equation}
    \pdv{}{x}(-\e^{\frac{2}{3}t} x^3) = -3\e^{\frac{2}{3}t} x^2 = -3\e^{\frac{2}{3}t}\frac{1}{3(\e^{\frac{2}{3}t} - 1) + 1} \to -1 \quad \text{as} \quad t \to \infty.
\end{equation}
Hence, the essential convergence rate is $f(x(t)) - f^\star = \Theta(\e^{-\frac{4}{3}t})$.

This example illustrates that it is impossible to recover the essential convergence rate by discretizing ODE~\eqref{x3} using step sizes constrained by the linear stability domain. As \eqref{nlin} shows, the step size within the stability region can increase only at a linear scale, whereas exponential growth is required to restore the essential convergence rate.

Interestingly, numerical experiments reveal that such aggressive step-size growth is in fact numerically feasible. In particular, line-search adaptations based on the Armijo condition lead to these aggressive choices. This phenomenon arises because ODE~\eqref{x3} is purely nonlinear, while step-size restrictions are derived from linear stability analysis.

Finally, it is important to emphasize that this failure does not contradict \cref{thm}. The theorem does not assert that $\alpha(t_{k_j}) = rk_j + \varepsilon k_j$, but only that $\alpha(t_{k_j}) \le rk_j + \varepsilon k_j$.


\subsection{Possibility of faster discrete rates than the essential convergence rate using step sizes depending on $f$}
Assumption~\ref{asm3} in~\cref{thm} states that the step size is chosen so that the linear stability holds for all $f \in \mathcal{F}$.
This is a strong requirement, and here we consider the possibility of acceleration if the step-size restriction is allowed to vary depending on each $f \in \mathcal{F}$.

We examine a scenario in which the functions that decrease exactly at the worst-case convergence rate differ from those that constrain the step sizes in \cref{thm}.
In such a case, it may be possible to accelerate the convergence rate.
As an example, consider the 1-essential AGM ODE~\eqref{NAGLODE} for $\mathcal{S}_{0,L}$ with $d=1$.
The essential convergence rate is $\Order{1/t^2}$, and this is tight. In fact, for $c > 3$, we consider $f \in \mathcal{S}_{0,L}$ defined by
\begin{equation}
       f(x) = \begin{cases}
            \frac{L(c-3)}{c(c-2)^2}\abs{x}^c  & \text{if} \quad  \abs{x} \le 1, \\
            \frac{L(c-3)}{c(c-2)^2} (c(\abs{x} -1) + 1) & \text{otherwise.} 
       \end{cases}
       \label{fattouch}
\end{equation}
For this $f$, ODE~\eqref{NAGLODE} becomes
\[
    \simulparen{
        \dot{x} &= \frac{2}{t}(v-x)\\
        \dot{v} &= -\frac{t}{2L}\nabla f(x),
    }\label{NAG2}
 \quad (x(0), v(0)) = (1,1)
\]
The solution is unique by~\cite[Theorem~1]{SBC16} and is given by
\[
    x(t) = \paren*{\frac{2}{t + 2}}^{\frac{2}{c-2}}.
\]
hence, the convergence rate is 
\begin{equation}
       f(x(t)) - f^\star = \frac{L(c-3)}{c(c-2)^2}\paren*{\frac{2}{t + 2}}^{\frac{2c}{c-2}}.
\end{equation}
The exponent $\frac{2c}{c-2}$ tends to 2 as $c \to \infty$ (a similar discussion can be found in~\cite{ACPR18}).
Therefore, the essential convergence rate is $\Omega(1/t^2)$.
However, in this case we observe that $\rho\left(\pdv{g[f]}{y}\right) = \Order{1/t}$, which implies that constant step sizes are overly conservative for this $f$ in light of linear stability analysis. This suggests that the convergence rate may be accelerated by increasing the step sizes.
On the other hand, quadratic functions such as $f(x) = \tfrac{L}{2}x^2$ are instances that restricts step sizes in \cref{thm}, and for these functions we cannot increase the step sizes. However, for quadratic functions, this ODE has a convergence rate $\Order{1/t^3}$~\cite{SBC16}.
This observation implies that if step-size restrictions arise only from quadratic functions, it may be possible to accelerate the discrete convergence rate to $\Order{1/k^3}$ by allowing the step sizes to depend on $f$.

Note that this does not contradict the lower bound $\Omega(1/k^2)$~\cite{D17}, since this bound is valid only for finite iterations.
Actually, there is a method that achieves a convergence rate of $\Order{1/k^{2.5}}$~\cite{JM23}.

\subsection{Possibility of faster discrete rates than the essential convergence rate using time-dependent numerical methods}
As illustrated in the examples, discretization can yield a constant-factor acceleration over the essential convergence rate. In \cref{thm}, this phenomenon is described by the factor $r$ in the theorem's conclusion. This implies that if a numerical method has a stability domain that is large in the direction of the eigenvalues of the Jacobian, the convergence rate may be improved by a constant factor.
Similar insights are reported in~\cite{SRBd17}, where AGM for strongly convex functions is interpreted as a linear multi-step numerical integration method, for the gradient flow. This numerical method exhibits high stability for the gradient flow, enabling large step sizes and yielding a constant-factor improvement in the exponent of the convergence rate compared to the other methods.

It is also possible to accelerate convergence rates beyond constant factors by using discretizations that do not satisfy assumption~\ref{asm1}.
If the stability domain expands as the iteration proceeds, larger step sizes can be employed, thereby achieving more than a constant-factor acceleration.
A related result is presented in~\cite{NSM24}, where the AGM for convex functions, which achieves the convergence rate of $\Order{1/k^2}$, is interpreted as a variable step-size linear multi-step method for the gradient flow, whose convergence rate is $\Order{1/t}$.
It is demonstrated that in this case the step-size restriction increases linearly with the iteration count, resulting in the acceleration of the convergence rate by a factor of $t$.

\section{Conclusions}
\label{sec:conclusions}
We propose the notion of the essential convergence rate, which eliminates the ambiguity of convergence rates and facilitates meaningful comparison of ODE models.
According to linear stability analysis, any acceleration of convergence rates achieved by time-rescaling ODEs is offset during discretization, ultimately yielding the essential convergence rate. 

\appendix
\section{Convergence rates of illustrative examples} \label{app:rates_ie}

\paragraph{Gradient flow for strongly convex quadratic functions}

For the $1$-dimensional problem $ \dot{x} = - \lambda x $ with the objective function $ f(x) = \frac{\lambda}{2} x^2 $, the solution can be written as $ x(t) = x(0) \e^{- \lambda t} $. Therefore, $ f (x(t)) - f^\star = \frac{\lambda}{2} \paren*{ x(0) }^2 \e^{- 2 \lambda t }$ holds. 
This implies that, when we consider $\mu$-strongly convex quadratic objective functions on $ \RR^d $, the convergence rate of the optimal gap is $\Order{\e^{-2\mu t}}$. 

\paragraph{Steepest descent for strongly convex quadratic functions}

The solution of the gradient descent $ \x{k+1} = \x{k} - h \lambda \x{k} $ for the $1$-dimensional objective function $ f(x) = \frac{\lambda}{2} x^2 $ can be written as $ \x{k} = ( 1 - h \lambda)^k \x{0} $. 
Therefore, for $L$-smooth and $\mu$-strongly convex quadratic objective functions on $ \RR^d $, the convergence rate of the optimal gap is $ \max_{ \lambda \in [\mu,L]} \abs*{ 1 - h \lambda }^{2k} f \paren*{\x{0} } $. 
The step size that optimizes the convergence rate is $ h = 2 / (L+\mu) $, and the rate is 
\[ f\paren*{ \x{k} } - f^\star \le \paren*{ 1 - \frac{2\mu}{L+\mu} }^{2k} f \paren*{\x{0} } = \Order{ \e^{ - 4 \frac{\mu}{L+\mu} k } }. \]

\section{Relationship between first-order reformulations of ODEs and the eigenvalues of the Jacobian of the vector fields}\label{app:firstorder}

We consider the following second-order ODE:
\begin{equation}
       \ddot{x}(t) + a_1(t) \dot{x}(t) +  a_2(t)\nabla f(x(t)) = 0,
\end{equation}
where $a_1(t),\, a_2(t) \in \mathbb{R}$ are possibly time-dependent coefficients.
The standard method for converting this into a first-order system is to introduce $v = \dot{x}$ and describe the ODE as a system for $x$ and $v$:
\begin{equation}
       \begin{cases}
              \dot{x}(t) = v(t),\\
              \dot{v}(t) = -a_1(t) v(t) - a_2(t)\nabla f(x(t)).
       \end{cases}\label{sysx}
\end{equation}
Here, we consider alternative first-order reformulations and examine whether the eigenvalues of the Jacobian of the resulting vector field remain invariant.

Using a possibly time-dependent invertible matrix $A(t) \in \mathbb{R}^{2d \times 2d}$, define $v_1(t) \in \mathbb{R}^d$ and $v_2(t) \in \mathbb{R}^d$ via:
\begin{equation}
       \begin{pmatrix}
              x(t)\\ \dot{x}(t)
       \end{pmatrix}
       =
       A(t)
       \begin{pmatrix}
              v_1(t)\\ v_2(t)
       \end{pmatrix}.
\end{equation}
Hereafter, we omit the argument $t$ for brevity.

Then, the ODE can be rewritten as a system for $v_1$ and $v_2$:
\begin{equation}
       \dv{}{t}
       \begin{pmatrix}
              v_1\\ v_2
       \end{pmatrix}
       = \dv{A^{-1}}{t} A
        \begin{pmatrix}
              v_1\\ v_2
       \end{pmatrix}
       + A^{-1}\!
       \begin{pmatrix}
              O_d & I_d\\
              O_d & -a_1 I_d
       \end{pmatrix}
       A
       \begin{pmatrix}
              v_1\\ v_2
       \end{pmatrix}
       - A^{-1}\!\!
       \begin{pmatrix}
              O_d\\ a_2\nabla f\left(\left[A\begin{pmatrix} v_1\\ v_2 \end{pmatrix}\right]_1\right)
       \end{pmatrix}, \label{sysv}
\end{equation}
where $[\cdot]_1$ denotes the first $d$ elements of a vector.

\begin{proposition}
        Suppose that $\dv{A^{-1}}{t} A \to O$ as $t \to \infty$.
        Then, in the limit as $t \to \infty$, the eigenvalues of the Jacobian of the right-hand side vector field of~\eqref{sysv} and \eqref{sysx} coincide.
\end{proposition}

\begin{proof}
       Let $A$ and $A^{-1}$ be written as block matrices:
       \begin{equation}
              A = \begin{pmatrix} A_{11} & A_{12}\\ A_{21} & A_{22} \end{pmatrix}, \quad A^{-1} = \begin{pmatrix} B_{11} & B_{12}\\ B_{21} & B_{22} \end{pmatrix},
       \end{equation}
       where each block is a $d \times d$ matrix.

       Then, in the limit as $t \to \infty$, the Jacobian of the vector field on the right-hand side of~\eqref{sysv} coincides with the limit of
       \begin{equation}
              A^{-1}
              \begin{pmatrix}
                     O_d & I_d\\
                     O_d & -a_1 I_d
              \end{pmatrix}
              A - 
              \begin{pmatrix}
                     B_{12} \\ B_{22}
              \end{pmatrix}
              a_2\nabla^2 f(A_{11}v_1 + A_{12}v_2)
              \begin{pmatrix}
                     A_{11} & A_{12}
              \end{pmatrix}.\label{appc3}
       \end{equation}
       Using the identity $A A^{-1} = I_{2d}$, we have:
       \begin{equation}
              A
              \begin{pmatrix}
                     B_{12} \\ B_{22}
              \end{pmatrix}
              =
              \begin{pmatrix}
                     O_d\\ I_d
              \end{pmatrix},\qquad
              \begin{pmatrix}
                     A_{11} & A_{12}
              \end{pmatrix}
              A^{-1}=
              \begin{pmatrix}
                     I_d & O_d
              \end{pmatrix}.
       \end{equation}
       Therefore, the eigenvalues of~\eqref{appc3} are the same as those of
       \begin{align}
              &\begin{pmatrix}
                     O_d & I_d\\
                     O_d & -a_1 I_d
              \end{pmatrix}
              -
              A
              \begin{pmatrix}
                     B_{12} \\ B_{22}
              \end{pmatrix}
              a_2\nabla^2 f(A_{11}v_1 + A_{12}v_2)
              \begin{pmatrix}
                     A_{11} & A_{12}
              \end{pmatrix}
              A^{-1}\\
              &\quad =
              \begin{pmatrix}
                     O_d & I_d\\
                     a_2\nabla^2 f(A_{11}v_1 + A_{12}v_2) & -a_1 I_d
              \end{pmatrix}
              =
              \begin{pmatrix}
                     O_d & I_d\\
                     a_2\nabla^2 f(x) & -a_1 I_d
              \end{pmatrix},
       \end{align}
       which is the Jacobian of the vector field in~\eqref{sysx}.
\end{proof}

From the proof, we can see that if we restrict ourselves to time-independent choices of $A$, the eigenvalues are invariant under changes of $A$.

In the case of \eqref{NAG}, we use a time-dependent matrix
\[
A(t) = \begin{pmatrix}
    1 & 0 \\
    -\dot{\alpha}(t) / \alpha(t) & \dot{\alpha}(t) / \alpha(t)
\end{pmatrix}
\]
to construct a first-order ODE from
\[
\ddot{x} + \left(\frac{2\dot{\alpha}}{\alpha} - \frac{\ddot{\alpha}}{\dot{\alpha}}\right)\dot{x} + \frac{(\dot{\alpha})^2}{2\alpha}\nabla f(x) = 0.
\]
The assumption $\dv{A^{-1}}{t} A \to O$ holds if $\dv{}{t} \log \frac{\alpha}{\dot{\alpha}} \to 0$, especially for the 1-essential time scale $\alpha(t) = t^2/L$.





\section*{Acknowledgments}
This work was supported by JSPS KAKENHI (24KJ0595).

\bibliographystyle{siamplain}
\bibliography{references}
\end{document}